\documentclass[10pt,a4paper]{article}
\usepackage[utf8]{inputenc}
\usepackage[colorlinks=true]{hyperref}
\usepackage{amsmath, amsthm}
\usepackage{amsfonts}
\usepackage{amssymb}
\usepackage{fixltx2e}
\usepackage{mathtools}
\usepackage{times}
\usepackage[hyperref, dvipsnames]{xcolor}
\hypersetup{
  colorlinks=true,
}

\usepackage{enumitem}

\usepackage[left=2cm,right=2cm,top=2cm,bottom=2cm]{geometry}
\usepackage{comment}
\newcommand{\grad}{\nabla}
\newcommand{\inc}{K}

\newcommand{\norm}[2]{\left\lVert #1\right\rVert_{#2}}
\newcommand{\weaklyto}{\rightharpoonup}

\newcommand{\Id}{\mathrm{I}}

\newcommand{\cts}{\hookrightarrow}
\newcommand{\ctsCompact}{\xhookrightarrow{c}}
\newcommand{\ctsDense}{\xhookrightarrow{d}}

\newcommand{\m}{\mathsf{m}}
\newcommand{\M}{\mathsf{M}}
\providecommand{\keywords}[1]{\textbf{\textit{Keywords: }} #1}

\newtheorem{theorem}{Theorem}[section]
\newtheorem{prop}[theorem]{Proposition}
\newtheorem{lem}[theorem]{Lemma}

\newtheorem{egs}[theorem]{Examples}
\newtheorem{eg}[theorem]{Example}
\newtheorem{remark}[theorem]{Remark}

\theoremstyle{definition}

\begin{document}
\hypersetup{
  urlcolor     = blue, 
  linkcolor    = Bittersweet, 
  citecolor   = Cerulean
}
\title{On the differentiability of the minimal and maximal solution maps of elliptic quasi-variational inequalities\footnotetext{The authors thank the referees for their excellent comments which helped to improve the paper. AA and MH were partially supported by the DFG through the DFG SPP 1962 Priority Programme \emph{Non-smooth and Complementarity-based
Distributed Parameter Systems: Simulation and Hierarchical Optimization} within project 10. 
MH and CNR acknowledge the support of Germany's Excellence Strategy - The Berlin Mathematics Research Center MATH+ (EXC-2046/1, project ID: 390685689) within project AA4-3.
In addition, MH acknowledges the support of SFB-TRR154 within subproject B02, and CNR was supported by NSF grant DMS-2012391.}}
\author{Amal Alphonse\thanks{Weierstrass Institute, Mohrenstrasse 39, 10117 Berlin, Germany ({\tt alphonse@wias-berlin.de})} \and Michael Hinterm\"{u}ller\thanks{Weierstrass Institute, Mohrenstrasse 39, 10117 Berlin, Germany ({\tt hintermueller@wias-berlin.de})}   \and Carlos N. Rautenberg\thanks{Department of Mathematical Sciences and the Center for Mathematics and Artificial Intelligence (CMAI), George Mason University, Fairfax, VA 22030, USA ({\tt crautenb@gmu.edu})}}
\maketitle
\abstract{In this note, we prove that the minimal and maximal solution maps associated to elliptic quasi-variational inequalities of obstacle type are directionally differentiable  with respect to the forcing term and for directions that are signed. Along the way, we show that the minimal and maximal solutions can be seen as monotone limits of solutions of certain variational inequalities and that the aforementioned directional derivatives can also be characterised as the monotone limits of sequences of directional derivatives associated to variational inequalities. We conclude the paper with some examples and an application to thermoforming.}

\keywords{Quasi-variational inequality, obstacle problem, directional differentiability, minimal and maximal solutions, ordered solutions.}

\section{Introduction}\label{sec:intro}
Quasi-variational inequalities (QVIs) are variational inequalities (VIs) where the constraint set over which the solution is sought also depends on the solution itself.  As such, QVI problems are highly nonlinear and nonconvex and in sharp contrast to the usual setting for VIs, QVIs usually possess multiple solutions. In certain situations, the set of solutions can be ordered in the sense that there exist minimal and maximal solutions. In this paper, we address the directional differentiability of the maps taking the source term of a QVI into the minimal and maximal solutions. The above-mentioned quirks of QVIs endow their study with substantial technical issues to overcome when examining questions of stability analysis and differential sensitivity.

QVIs were first formulated by Bensoussan and Lions \cite{BensoussanLionsArticle, Lions1973} in the modelling of stochastic impulse controls.  Applications of QVIs are ubiquitous. Among some, we mention thermoforming processes \cite{AHR, ARJF}, the formation and growth of lakes, rivers and sandpiles \cite{PrigozhinSandpile,BarrettPrigozhinSandpile,Prigozhin1994,MR3231973}, generalised Nash equilibrium games \cite{HARKER199181, Facchinei2007, Pang2005}, and magnetisation of superconductors \cite{KunzeRodrigues,BarrettPrigozhinSuperconductivity,Prigozhin,MR1765540}. Additional details and references can be found in our survey paper \cite{AHRTrends} and the book \cite{Baiocchi}.

We focus on elliptic QVIs of obstacle type (these are also known as implicit obstacle problems). The precise formulation is as follows. Let $V \subset H$ be a continuous and dense embedding of separable Hilbert spaces and denote the inner product on $H$ by $(\cdot,\cdot) = (\cdot,\cdot)_H$. Suppose that there exists an ordering to elements of $H$ via a closed convex cone $H_+$ that satisfies 
\[H_+ = \{h \in H : (h,g) \geq 0 \quad \forall g \in H_+\}.\]
The ordering is defined by: $h_1 \leq h_2$ if and only if $h_2-h_1 \in H_+$. This endows an ordering  for $V$ in the obvious way and we write $V_+ := \{ v \in V : v \geq 0\}$.  It also induces one for the dual space $V^*$ via
\[V^*_+ := \{ f \in V^* : \langle f, v\rangle \geq 0 \quad \forall v \in V_+\},\]
where $\langle \cdot, \cdot \rangle = \langle \cdot, \cdot \rangle_{V^*, V}$ is the standard duality product. We write $h^+$ for the orthogonal projection of $h \in H$ onto the space $H_+$ and we use the decomposition $h = h^+ - h^-$. Suppose that $v \in V$ implies that $v^+ \in V$ and that there exists a constant $C >0$ such that $\lVert{v^+}\rVert_{V} \leq C\norm{v}{V}$ for all $v \in V$.\footnote{Some concrete realisations of this setup will be presented in Examples \ref{eg:examples}.}

Let  $A\colon V \to V^*$ a linear operator that satisfies the following properties for all $u, v \in V$: 
\begin{align*}
\langle Au, v \rangle &\leq C_b\norm{u}{V}\norm{v}{V}\tag{boundedness},\\
\langle Au, u \rangle &\geq C_a\norm{u}{V}^2\tag{coercivity},\\
\langle Au^+, u^- \rangle &\leq 0\tag{T-monotonicity}.
\end{align*}
We consider an obstacle map
\[\Phi\colon H \to V\]
and the associated constraint set $\mathbf{K}\colon H \rightrightarrows V$ defined by 
\[\mathbf{K}(y):=\{v \in V : v \leq \Phi(y)\}.\]
Given $f \in V^*$, we study QVIs of the form
\begin{equation}\label{eq:QVI}
\text{find } y \in \mathbf{K}(y) : \langle Ay-f, y-v \rangle \leq 0 \quad \forall v \in \mathbf{K}(y).
\end{equation}
In this work, we always take $\Phi$ to be increasing, i.e., 
\[u \leq v \text{ implies } \Phi(u) \leq \Phi(v),\]
and we suppose also that (unless stated otherwise) it satisfies the property
\begin{align}
&\text{if $\{v_n\} \subset V$ is decreasing with $v_n \weaklyto v$ in $V$ and $v \leq \Phi(v_n)$, then $v \leq \Phi(v).$}\label{ass:PhiPresOrder}
\end{align}
This holds, for example, if $\Phi$ is linear and continuous, or if $\Phi\colon V \to V$ is weakly sequentially continuous (that is,  $v_n \weaklyto v$ in $V$ implies $\Phi(v_n) \weaklyto \Phi(v)$ in $V$).

We define $\mathbf{Q}$ to be the source-to-solution map associated to \eqref{eq:QVI}, so that the inequality reads $y \in \mathbf{Q}(f)$.  To prove that the set $\mathbf{Q}(f)$ is non-empty (and indeed to properly define the problem under study in this article) we need some additional details. 
\subsection{Existence of (extremal) solutions}\label{sec:existenceOfExtremals}
Fixing an obstacle $\varphi \in V$, consider the VI
\[u \in \mathbf{K}(\varphi) : \langle Au - f, u - v \rangle \leq 0 \quad \forall v \in \mathbf{K}(\varphi)\]
and denote its solution map $S\colon V^* \times H \to V$ so that the above can be stated as $u=S(f,\varphi)$. It follows that $\mathbf{Q}(f)$ is the set of fixed points of $\varphi \to S(f,\varphi)$. In order to show the presence of fixed points, we are going to assume the existence of a \emph{subsolution} $\underline{u}$ and a \emph{supersolution} $\overline{u}$  for $S(f,\cdot)$, that is,
\[\exists \underline{u}, \overline{u} \in V \quad \text{such that}\quad \underline{u} \leq S(f,\underline{u}) \quad\text{and}\quad \overline{u} \geq S(f, \overline{u}),\]
and such that the subsolution lies below the supersolution:
\[\underline{u} \leq \overline{u}.\]
\begin{remark}\label{rem:subAndSupSolns}
For a supersolution, we can take any $\overline{u} \in V$ satisfying $\overline{u} \geq A^{-1}f$ where the right-hand side is (by definition) the solution of the equation $Az = f.$ This is a valid choice since $A^{-1}f$ can be considered as the solution\footnote{Informally, $A^{-1}f = S(f, \Phi^{-1}(\infty))$.} of the VI with source $f$ and obstacle equal to $\infty$, which by the comparison principle \cite[Theorem 5.1, \S 4:5]{Rodrigues}, exceeds $S(f,\overline{u})$. If $f \geq 0$ and $\Phi(0) \geq 0$, then we may take $\underline{u}=0$ to be a subsolution: $0 = S(0,0) \leq S(f,0)$. 
\end{remark}
Under these circumstances, we can apply the Birkhoff--Tartar theory \cite{Tartar74, Birkhoff} (see also \cite[Chapter 15.2.2]{Aubin} and \cite[Chapter 2.5]{Mosco1976}) of fixed points in vector lattices to obtain not only that
\[\mathbf{Q}(f) \cap [\underline{u},\overline{u}] \neq \emptyset\]
(i.e., \eqref{eq:QVI} has solutions on $[\underline{u}, \overline{u}]$), but moreover, there exists a minimal solution $\m(f)$ and a maximal solution $\M(f)$ in this interval with respect to the ordering introduced above. 
These satisfy
\[\m(f) \leq u \leq \M(f) \quad \forall u \in \mathbf{Q}(f) \cap [\underline{u}, \overline{u}].\]
\begin{remark}
The condition \eqref{ass:PhiPresOrder} is not required for the above results.
\end{remark}
One should bear in mind that these maps $\m$ and $\M$ are associated a particular interval $[\underline{u}, \overline{u}]$ and sometimes we may show the dependence on the interval explicitly by writing  $\m_{[\underline u, \overline u]}
$ and $\M_{[\underline u, \overline u]}$, or more generally, provided they exist, we denote the minimal and maximal solutions on some interval $A$ by $\m_A$ and $\M_A$ respectively. 

However, when discussing the minimal solution map, only the lower bound of the interval under consideration is pertinent due to the trivial observation that  if $w \leq a \leq b$, $\m|_{[w, a]}(f) = \m|_{[w, b]}(f)$ whenever these exist, and similarly, only the upper bound of the interval for the maximal solution map is important since if $a \leq b \leq z$ then $\M|_{[a,z]}(f) = \M|_{[b,z]}(f)$.
\subsection{Aim of the article}\label{sec:aims}
In this paper, we are interested in the directional differentiability of $f \mapsto \m(f)$ and $f \mapsto \M(f)$. We will show that, under some assumptions, these maps are indeed directionally differentiable for a subset (that we will specify below) of directions belonging to $V^*$. That is, we prove the existence of the following limits: 
\begin{align}
\lim_{s \to 0^+}\frac{\m(f+sd)-\m(f)}{s} \qquad\text{and}\qquad \lim_{s \to 0^+}\frac{\M(f+sd)-\M(f)}{s}.
\end{align}
We also give a characterisation of the derivative in terms of a related QVI. This builds upon our previous works \cite{AHR, AHROC, AHRStability} in two ways. In \cite{AHR, AHROC} we showed that $\mathbf{Q}$ has a directional derivative in the sense that we proved that for every $u \in \mathbf{Q}(f)$, given a direction $d \in V^*$, there exists $u^s \in \mathbf{Q}(f+sd)$ and (a directional derivative) $\alpha \in V$ such that 
\[\lim_{s \to 0^+} \frac{u^s-u}{s} = \alpha\]
(with the limit taken in $V$).  In \cite{AHROC}, we derived further existence results for \eqref{eq:QVI} and procedures to iteratively approach solutions of the QVI. Furthermore, we also obtained stationarity systems for optimal control problems with QVI constraints. On the other hand, in \cite{AHRStability}, we studied continuity properties of the extremal operators $\m$ and $\M$. This work then can be considered as a bridge between these two sets of papers.

The motivations for this study are manifold:
\begin{enumerate}[label=(\roman*)]
\item the mathematical problem itself is challenging and interesting,
\item suitable characterisations of the directional derivative could be used to develop efficient
bundle-free numerical solvers and solution algorithms along the lines of \cite{MR3555389},
\item in applications, it can be important to know the influence of the source term on the solution, for example in thermoforming (which can be modelled through QVIs, see \cite[\S 6]{AHR} and also \S \ref{sec:thermo} later), manufacturers may wish to understand the sensitivity of the desired shapes or products to changes in heating processes under their control,
\item for optimal control problems with control-to-state maps involving $\m$ or $\M$, differentiability of the maps is necessary in order to state certain first-order stationarity conditions satisfied by the optimal control,
\item in applications involving optimal control problems with QVI constraints, as typically there are many states associated to a single (optimal) control (due to non-uniqueness of solutions), it can be important to minimise, for instance, the difference $\M(f)-\m(f)$. In the case of thermoforming, manufacturers may wish to reproduce shapes or products that are within a certain acceptable tolerance value. 
\end{enumerate}
Note that continuity properties of these maps (studied under a similar functional setting in \cite{AHRStability}) are vital for the existence results for optimal control problems.

The idea is to base our developments on the differentiability results obtained in \cite{AHROC}; let us recall this and set the scene in the next section. First, some words on our notation and conventions.

\medskip

\noindent \textbf{Notation.} Throughout the rest of the paper, we shall use the notation $o(\cdot)$ to denote a remainder term, i.e., $s^{-1}o(s) \to 0$ in $V$ as $s \to 0^+$. We will write $v_n \nearrow v$ to mean $v_n \to v$ and $v_{n+1} \geq v_n$ for all $n$ (i.e., $\{v_n\}$ is an increasing sequence that converges to $v$). Similarly, $\searrow$ will refer to decreasing convergent sequences.
The notation $B_R(z)$ will be used to mean the closed ball in $V$ of radius $R$ centred at $z$.
\section{Background on directional differentiability for QVIs}
Recall from \S \ref{sec:intro} the constraint set mapping $\mathbf{K}\colon H \rightrightarrows V$ 
which is closed (since $V_+$ is closed) and convex valued, and associated to this, we define the \emph{tangent cone} of $\mathbf K(w)$ at a point $u\in \mathbf K(w)$ by
\[\mathcal{T}_{\mathbf K(w)}(u) := \overline{\mathcal{R}_{\mathbf K(w)}(u)} \quad\text{where}\quad \mathcal{R}_{\mathbf K(w)}(u) := \{ h \in V : \exists  s^* > 0 \text{ such that } u+sh \in \mathbf{K}(w) \;\forall s \in [0,s^*]\}.\]
The latter is known as the \emph{radial cone} and can be regarded as the set of all feasible directions. See for example \cite[\S 2.2.4]{Bonnans} for more details on these concepts.

 A fundamental result that we shall need is the following, which, under certain circumstances, tells us that $\mathbf{Q}$ has a directional derivative in a certain sense and provides a characterisation. 
%
\begin{theorem}[{\cite[Theorem 3.2]{AHROC}}]\label{thm:dirDiff1}
Given $f \in V^*$ and $d \in V^*$, for every $u \in \mathbf{Q}(f)$, under the local assumptions 
\begin{align}
&\text{there exists $\epsilon > 0$ such that  } \Phi\colon B_\epsilon(u) \to V \text{ is Lipschitz with Lipschitz constant $C_\Phi < C_a\slash(C_a+C_b)$},\label{ass:PhiLipBound}\\
&\text{$\Phi\colon V \to V$ is Hadamard directionally differentiable at $u$,}\label{ass:PhiHadamardAtY}
\end{align}
there exists $u^s \in \mathbf{Q}(f+sd)\cap B_R(u)$  (where $0 < R \leq \epsilon$) and $\alpha=\alpha(d)\in V$ such that
\[u^s = u + s\alpha + o(s),\]
where  $\alpha$ satisfies the QVI
\begin{equation}\label{eq:QVIforAlphaGeneral}
\begin{aligned}
\alpha \in \mathcal{K}^u(\alpha) &: \langle A\alpha - d, \alpha - v \rangle \leq 0 \quad \forall v \in \mathcal{K}^u(\alpha),\\
\mathcal{K}^u(\alpha) &:= \Phi'(u)(\alpha)  + \mathcal{T}_{\mathbf K(u)}(u) \cap [f-Au]^\perp.   
\end{aligned}
\end{equation}
The directional derivative $\alpha=\alpha(d)$ is positively homogeneous in $d$.  
\end{theorem}
It is worth emphasising that in fact, \cite[Theorem 3.2]{AHROC} is posed in a much more general functional setting than under consideration in this paper (described in the introduction).
\begin{remark}\label{rem:2}
It is not too difficult to see that \cite[Lemma 3.12]{AHROC}  if $\Phi$ is Hadamard differentiable on $B_\epsilon(u)$, then the condition
\begin{align*}
&\exists \epsilon > 0 : \norm{\Phi'(z)(v)}{V} \leq C_\Phi\norm{v}{V}\quad\forall z \in B_\epsilon(u), \forall v \in V,\text{ where } C_\Phi < C_a\slash(C_a+C_b)\label{ass:smallnessOfDerivOfPhiAty}
\end{align*}
implies \eqref{ass:PhiLipBound}; this can in certain situations be easier to verify.
\end{remark}
The set $\mathcal{K}^u(\alpha)$ appearing in \eqref{eq:QVIforAlphaGeneral} is called a \emph{critical cone} and we can be a bit more precise about its structure if  additional assumptions on the functional setting are available as the next remark shows.
\begin{remark}[Dirichlet space case]\label{rem:dirichletSpace}
Suppose that $H:=L^2(X;\mu)$ where $X$ is a locally compact, separable metric space and $\mu$ is a positive Radon measure on $X$ with full support\footnote{This means that $\mu$ is a non-negative Borel measure which is finite on compact sets and strictly positive on non-empty open sets.}, and let $V \subset H$ be a dense subspace (the ordering is given by the usual a.e. ordering of functions). 

Assume that 
\begin{enumerate}[label=(\roman*)]
\item there exists a symmetric, positive semidefinite bilinear form $\xi\colon V \times V \to \mathbb{R}$ such that $V$ is a Hilbert space when  endowed with
\[(\cdot,\cdot)_V := (\cdot,\cdot)_H + \xi(\cdot,\cdot),\]
\item $\text{if $u \in V$ then $\hat u := \min(u^+,1) \in V$ and $\xi(\hat u,\hat u) \leq \xi(u,u)$}$
(this is the \emph{Markov property}),
\item the following density conditions hold:
\begin{equation*}
V\cap C_c(X)  \ctsDense C_c(X) \qquad \text{and}\qquad  V\cap C_c(X)  \ctsDense V. 
\end{equation*}
\end{enumerate}
Under these circumstances, $V$ is known as a \emph{regular Dirichlet space} and the notions of capacity, quasi-continuity and related concepts are well defined, see \cite[\S 2.1]{FukushimaBook} and \cite[\S 3]{Haraux1977} for more details.  Furthermore, from the work\footnote{More specifically, we have the polyhedricity of sets of obstacle type and the differentiability of VI solution maps associated to these types of constraint sets in \cite[Theorem 3.3]{MR0423155}.} of Mignot  \cite[Theorem 3.2, Lemma 3.2]{MR0423155}, the critical cone appearing in \eqref{eq:QVIforAlphaGeneral} can be expressed explicitly as
\[\mathcal{K}^u(w) = \{ \varphi \in V : \varphi \leq \Phi'(u)(w) \text{ q.e. on }\mathcal{A}(u) \text{ and } \langle Au-f, \varphi -\Phi'(u)(w)\rangle = 0\}.\]
Here, `q.e.' means quasi-everywhere, i.e., everywhere except on a set of capacity zero, and $\mathcal{A}(u)$ refers to the \emph{active} or \emph{coincidence set}\footnote{Strictly speaking, we take the quasi-continuous representatives of the functions appearing in the definition of the active set so that $\mathcal{A}(y)$ is quasi-closed and defined up to sets of capacity zero.} of the solution $y$ to the QVI, i.e.,
\[\mathcal{A}(u) := \{ x \in X : u(x) = \Phi(u)(x)\} \quad \text{for $u \in V$.}\]
\end{remark}
Note that in the context of $V$ and $H$ being Sobolev or Lebesgue spaces over a domain $\Omega$, in certain cases the space $X=\overline{\Omega}$ and not merely $\Omega$! See Examples \ref{eg:examples} for more details.
\section{Strategy of proof and some preliminary results}\label{sec:strategy}

Let us sketch the procedure we will undertake for showing directional differentiability for the minimal solution map $\m$. Our aim is to show that, given a source term $f$ and a direction $d$, there exists an element $\m'(f)(d)$ such that
\[\m(f+sd)=\m(f) + s\m'(f)(d) + o(s).\]
As we have seen, Theorem \ref{thm:dirDiff1} states that under certain assumptions, given $u \in \mathbf{Q}(f)$, there exists $u^s \in \mathbf{Q}(f+sd)$ and $\alpha \in V$ such that
\[u^s = u + s\alpha + o(s).\]
We may select $u$ to be the minimal solution $\m(f)$ and it remains to prove that the selection mechanism of the theorem that furnishes the $u^s$ is indeed such that $u^s \equiv \m(f+sd)$. To do this, we need to take a closer look at the method of proof of the cited theorem. The proof relies on 
\begin{enumerate}
\item[(i)] creating an iterative sequence of solutions of VIs:
\begin{align*}
u^s_n &= S(f+sd, u^s_{n-1}),\\
u^s_0 &= u,
\end{align*}
(recall the definition of $S$ from \S \ref{sec:existenceOfExtremals})
\item[(ii)]  obtaining, by applying the sensitivity results for VIs \cite{MR0423155, WGuidedTour}, 
expansion formulas of the type 
\begin{equation}
u^s_n = u + s\alpha_n + o_n(s)\label{eq:expansionFormulaN}
\end{equation}
for these elements, and then 
\item[(iii)] passing to the limit $n \to \infty$ and identifying the limits of $\{u_n^s\}, \{\alpha_n\}$ and $\{o_n\}$.
\end{enumerate}
Thus, it is clear that we need to show that the limit of $\{u^s_n\}$ (which exists) is indeed $\m(f+sd)$. For this purpose, we need to prove some properties of $\m$ and $\M$ which we shall do so in a series of lemmas. Before we begin, it is important to record that in \cite[Proposition 3.7]{AHROC}, under the local Lipschitz assumption \eqref{ass:PhiLipBound}, we showed (via a Banach fixed point theorem argument) that 
\[u_n^s \to u^s \text{ in $V$  where } u^s \in \mathbf{Q}(f+sd) .\]
\subsection{Monotone convergence to extremal solutions}
Let $f \in V^*$ be a source term. In this section, we assume that we are given a subsolution $\underline{u}$ and a supersolution $\overline{u}$ for $S(f,\cdot)$ with $\underline{u} \leq \overline{u}$ (this is the same setup as in \S \ref{sec:existenceOfExtremals}). We introduce the assumption
\begin{align}
&\exists v_0 \in V : v_0 \leq \Phi(\underline{u}).\label{ass:feasibleForMin}
\end{align}
\begin{remark}
The condition \eqref{ass:feasibleForMin} essentially asks for $\mathbf{K}(\underline{u})$ to be non-empty. A typical situation is where $f \in V^*_+$ is non-negative and $\Phi$ is taken such that $\Phi(0) \geq 0$, in which case $\underline{u}:=0$ is a subsolution (as shown in Remark \ref{rem:subAndSupSolns}) so that $v_0 \equiv 0$ is a possibility.
\end{remark}
The following lemma demonstrates that, under some compactness on $\Phi$, starting at a subsolution we can converge to the minimal solution.  The result is not strictly necessary for our main theorem yet we provide it for interest and completeness (note that the required condition \eqref{ass:PhiCompletelyCts} is much stronger than \eqref{ass:PhiPresOrder}).
\begin{lem}\label{lem:convergenceToMin}Assume  \eqref{ass:feasibleForMin} and  that
\begin{equation}
\text{$\Phi\colon V \to V$ is completely continuous}\label{ass:PhiCompletelyCts}.
\end{equation}  
Then the sequence
\begin{equation}\label{eq:sequenceUn}
\begin{aligned}
u_n &:= S(f,u_{n-1}),\\
u_0&:=\underline{u}.
\end{aligned}
\end{equation}
satisfies $u_n \nearrow \m(f)$ in $V$.
\end{lem}
\begin{proof}
By definition, $u_0 \leq \m(f)$. By definition of subsolution and by using the comparison principle \cite[Theorem 5.1, \S 4:5]{Rodrigues}, $u_0 \leq S(f,u_0) = u_1 \leq S(f,\m(f)) = \m(f)$. Arguing in a similar way, $u_0 \leq u_n \leq u_{n+1} \leq \m(f)$ for all $n$.

Since $\Phi$ is increasing, it follows that $v_0 \leq \Phi(u_n)$ for each $n$. Hence, we may test the VI (for $u_n$) 
\[\langle Au_n - f, u_n - v \rangle \leq 0 \quad \forall v \in V : v \leq \Phi(u_{n-1})\]
with $v_0$ and use basic manipulations to obtain a uniform bound on $u_n$, which in combination with the fact that $\{u_n \}$ is monotonic, leads to the existence of some $u \in V$ such that $u_n \weaklyto u$ in $V$ for the full sequence \cite[Lemma 2.3]{AHROC}. Now, in the VI for $u_n$, test with $v_n:=v^*-\Phi(u)+\Phi(u_{n-1})$ where $v^* \in V$ satisfies $v^* \leq \Phi(u)$. This is clearly a feasible test function such that $v_n \to v^*$ in $V$ (strongly by \eqref{ass:PhiCompletelyCts}), and we can pass to the limit in the resulting inequality to find exactly that $u \in \mathbf{Q}(f)$.
It follows also that $u \in [\underline{u},\m(f)]$ and therefore $u=\m(f)$. The strong convergence of $\{u_n\}$  is a result of \eqref{ass:PhiCompletelyCts} along with a standard continuous dependence estimate (e.g., see \cite[Equation (21)]{AHR}) 
\[\norm{u_n-u}{V} \leq C\norm{\Phi(u_{n-1})-\Phi(u)}{V}.\qedhere\]
\end{proof}
In \eqref{eq:sequenceUn}, if we instead start with the initial iterate at a supersolution of $S(f,\cdot)$, we are able to provide analogous results. 
Indeed, a similar argument to the proof of Lemma \ref{lem:convergenceToMin} proves the next lemma but note that complete continuity of $\Phi$ is not needed thanks to the structure of the unilaterally bounded constraint sets under consideration.
\begin{lem}\label{lem:uNConvergenceMax}Assume \eqref{ass:feasibleForMin}. 
Then the sequence
\begin{equation}\label{eq:sequenceUnMax}
\begin{aligned}
u_n &:= S(f,u_{n-1}),\\
u_0&:=\overline{u},
\end{aligned}
\end{equation}
satisfies $u_n \searrow \M(f)$ weakly in $V$. The convergence is strong under \eqref{ass:PhiCompletelyCts}.
\end{lem}
\begin{proof}
By definition, $u_0 \geq \M(f)$. By definition of supersolution and by using the comparison principle, $u_0 \geq S(f,u_0) = u_1 \geq S(f,\M(f)) = \M(f)$ and hence $u_0 \geq u_n \geq u_{n+1} \geq \M(f)$ for all $n$.

Clearly, $u_n \geq \underline{u}$, thus  
it follows from \eqref{ass:feasibleForMin} that $v_0 \leq \Phi(u_n)$ for each $n$. Hence, we may test the VI for $u_n$ with $v_0$ to obtain a uniform bound on $u_n$ and therefore we have $u_n \weaklyto u$ in $V$ for some $u$ (due to the monotonicity of $\{u_n\}$, as in the previous proof).

Now, take an arbitrary $v^* \in V$ with $v^* \leq \Phi(u)$. Since $\{u_n\}$ is decreasing, we have $u_n \leq u_j$ for all $n \geq j$ with $j$ fixed, i.e., $u_j-u_n \in V_+$. Passing to the limit $n \to \infty$ (using the fact that $V_+$ is closed and convex, hence weakly sequentially closed), we find $u_j - u \in V_+$ so that $u \leq u_j$. Using the fact that $\Phi$ is increasing, we see that $v^* \leq \Phi(u_j)$ for any $j \in \mathbb{N}$, i.e., $v^*$ is a valid test function for the VI for $u_n$. Taking the limit in the VI for $u_n$ allows us to deduce that $u$ satisfies
\[\langle Au-f, u - v \rangle \leq 0 \quad \forall v \in V : v \leq \Phi(u).\]
To conclude that $u \in \mathbf{Q}(f)$, we still need to show that $u \leq \Phi(u)$, but this is a consequence of \eqref{ass:PhiPresOrder}.

It follows also that $u \in [\M(f), \overline{u}]$ and therefore $u=\M(f)$.  The statement regarding strong convergence follows again by the continuous dependence estimate as above.
\end{proof}

\section{The minimal solution map}\label{sec:min}
We will show in this section that the minimal solution map is directionally differentiable under some assumptions by utilising the strategy explained in the previous section. 

As before, we take $\underline{u}$ to be a subsolution for $S(f,\cdot)$. Let $s \geq 0$ be a small parameter and take a direction $d \in V^*_+$. Since $\underline{u} \leq S(f,\underline{u}) \leq S(f+sd, \underline{u})$, $\underline{u}$ is also a subsolution for $S(f+sd,\cdot)$. In the other direction, we suppose, as a standing assumption in this section, that 
\begin{equation}\label{ass:supersolutionForPerturbedProblem}
\text{there exists $s_0 > 0$ such that for all $s \leq s_0$, $\overline{u}$ is a supersolution for $S(f+sd, \cdot)$}
\end{equation}
($\bar u$ should remain independent of $s$). 
Then, by the argument in \S \ref{sec:existenceOfExtremals}, we have the non-emptiness of the set $\mathbf{Q}(f+sd)\cap [\underline{u}, \overline{u}]$. Furthermore,  $\m(f+sd)$ exists on $[\underline{u}, \overline{u}]$ (for small $s$).
\begin{remark}\label{rem:1}
Asking for $\overline{u}$ to be a supersolution for the perturbed problem (parametrised by $s$) is not a stringent requirement since any supersolution for $S(f+sd,\cdot)$ is also a supersolution for $S(f+rd,\cdot)$ for $s \geq r \geq 0$, and because we will be sending $s \to 0$, we may always start by taking, for example, $\overline{u}$ to be a supersolution of $S(f+d,\cdot)$.
\end{remark}
\begin{eg}\label{eg:ofSupersolution}
As indicated in Remark \ref{rem:subAndSupSolns}, a concrete choice of a supersolution satisfying \eqref{ass:supersolutionForPerturbedProblem} is 
\[\overline{u} = A^{-1}(f+d).\]
\end{eg}
\begin{lem}\label{lem:NEWminComparison}Let $d \in V^*_+$ and  \eqref{ass:feasibleForMin} 
hold. Then $\m(f+sd) \geq \m(f)$.
\end{lem}
\begin{proof}
If $\Phi$ is, in addition, completely continuous, this is easy to show. Indeed, defining
\begin{align*}
y^s_n&:=S(f+sd,y^s_{n-1}),\\
y^s_0 &:= \underline{u},
\end{align*}
and with $u_n$ defined as in \eqref{eq:sequenceUn}, we see that $y^s_1 \geq u_1$ since $d \geq 0$. This implies that $y^s_n \geq u_n$ and hence, passing to the limit with the help of Lemma \ref{lem:convergenceToMin} , we have $\m(f+sd) \geq \m(f)$.

Let us prove the lemma (without the additional assumption). Define the sequence
\begin{align*}
y_n &:= S(f, y_{n-1}),\\
y_0 &:= \m(f+sd),
\end{align*}
and note that $y_1=S(f,y_0) \leq S(f+sd, y_0) = y_0$. This implies that $\{y_n\}$ is a decreasing sequence.

Since $y_0$ is the minimal solution on $[\underline{u}, \overline{u}]$, $y_0 \geq \underline{u}$ and from $y_1=S(f, y_0) \geq S(f,\underline{u}) \geq \underline{u}$ (by definition of subsolution), each $y_n \geq \underline{u}$.  This implies that $v_0$ from assumption \eqref{ass:feasibleForMin} satisfies $v_0 
\leq \Phi(y_n)$ for each $n$, and hence is a feasible test function for the VI for $y_n$. Just like in the proof of Lemma \ref{lem:convergenceToMin}, testing with $v_0$, we easily obtain a uniform bound on $\{y_n\}$ in $V$, and therefore we have $y_n \weaklyto y$ in $V$ for some $y$, which in addition to the fact that $\{y_n \}$ is decreasing, leads, via \eqref{ass:PhiPresOrder}, to
\[y_n \weaklyto y \in \mathbf{Q}(f)\]
in exactly the same way as in the proof of Lemma \ref{lem:uNConvergenceMax}. 


Furthermore, we have seen that $y \in [\underline{u}, y_0] = [\underline{u}, \m(f+sd)] \subset [\underline{u}, \overline{u}]$. This, along with $y \in \mathbf{Q}(f)$ implies that $\m(f) \leq y \leq \m(f+sd)$.
\end{proof}

Now, let us define (as sketched at the start of \S \ref{sec:strategy}) a sequence starting at $\m(f)$ with perturbed source term as follows:
\begin{equation}\label{eq:sequenceUsn}
\begin{aligned}
u^s_n &:= S(f+sd, u^s_{n-1}),\\
u^s_0 &:= \m(f).
\end{aligned}
\end{equation}
It is worth remarking that since $\m(f)$ acts as a subsolution for $S(f+sd,\cdot)$, if $\Phi$ is completely continuous, $u_n^s \nearrow \m|_{[\m(f),\overline{u}]}(f+sd)$ by Lemma \ref{lem:convergenceToMin} (the notation $\m|_A(f+sd)$ refers to the minimal solution on $A$). But in fact, the limit is the minimal solution on the full interval $[\underline{u}, \overline{u}]$ as the next result shows (which, we note, dispenses with the need for complete continuity). 
\begin{prop}\label{prop:usnConvMin}Let $d \in V^*_+$, \eqref{ass:feasibleForMin} and 
\begin{equation}
\text{there exists $\epsilon > 0$ such that  } \Phi\colon B_\epsilon(\m(f)) \to V \text{ is Lipschitz with Lipschitz constant $C_\Phi < C_a\slash(C_a+C_b)$}\label{ass:PhiLipBoundm}
\end{equation}
hold. Then $u^s_n \nearrow \m(f+sd)$ in $V$.
\end{prop}
\begin{proof}
We have $u_n^s \to u^s$ in $V$ for some $u^s \in \mathbf{Q}(f+sd)$ by \cite[Proposition 3.7]{AHROC}, thanks to \eqref{ass:PhiLipBoundm} (as explained at the start of \S \ref{sec:strategy}). Since $d$ is non-negative, the sequence $\{u_n^s\}$ is increasing 
and hence $u^s \in [\m(f), \infty)$, but we also see that $u^s_1 = S(f+sd, \m(f)) \leq S(f+sd, \m(f+sd))=\m(f+sd)$ (with the inequality by Lemma \ref{lem:NEWminComparison}) and thus $u^s \in [\m(f), \m(f+sd)] \subset [\underline{u}, \m(f+sd)]$. The claim follows.
%
%
\end{proof}
In the above proposition, if $\Phi\colon V \to V$ is completely continuous, then \eqref{ass:PhiLipBoundm} can be omitted. 
With all the preparations complete, we are ready to state the differentiability result.

\begin{theorem}[Directional differentiability of $\m$]\label{thm:min}
For $f \in V^*$, let $\underline{u}$ be a subsolution for $S(f,\cdot)$ and given $d \in V^*_+$, let $\overline{u}$ satisfy \eqref{ass:supersolutionForPerturbedProblem} with $\underline{u} \leq \overline{u}$.
 In addition to \eqref{ass:feasibleForMin} 
and \eqref{ass:PhiLipBoundm}, assume that 
\begin{align}
&\text{$\Phi\colon V \to V$ is Hadamard directionally differentiable at $\m(f)$.}\label{ass:PhiHadamardAtYm}
\end{align}
Then the map $\m\colon V^* \to V$ is directionally differentiable at $f$ in the direction $d$:  
\begin{equation*}
\lim_{s \to 0^+}\frac{\m(f+sd)-\m(f)}{s} = \m'(f)(d).
\end{equation*}
Furthermore, $\m'(f)(d)$ satisfies the QVI
\begin{equation}\label{eq:QVIforAlpha}
\begin{aligned}
\alpha \in \mathcal{K}_\m(\alpha) &: \langle A\alpha - d, \alpha - v \rangle \leq 0 \quad \forall v \in \mathcal{K}_\m(\alpha),\\
\mathcal{K}_\m(\alpha) &:= \Phi'(\m(f))(\alpha)  + \mathcal{T}_{\mathbf K(\m(f))}(\m(f)) \cap [f-A\m(f)]^\perp.  
\end{aligned}
\end{equation}			
\end{theorem}
\begin{proof}
The proof is as described at the start of this section. Indeed, a straightforward application of Theorem \ref{thm:dirDiff1} gives the existence of $\alpha \in V$ satisfying \eqref{eq:QVIforAlphaGeneral} such that the limit $u^s$ of $\{u^s_n\}$ (defined in \eqref{eq:sequenceUsn}) satisfies
\[u^s = \m(f) + s\alpha + o(s).\]
Proposition \ref{prop:usnConvMin} can be applied and it tells us that $u^s = \m(f+sd)$.
\end{proof}
We reiterate that assumption  \eqref{ass:supersolutionForPerturbedProblem} in the above theorem is satisfied by simply taking $\overline{u} = A^{-1}(f+d)$ as explained in Example \ref{eg:ofSupersolution} and hence this is, in some sense, not an assumption at all; we include it for full generality and in order to have $\m$ clearly defined (recall that $\m$ and $\M$ require the identification of an ordered interval to be sensible).

The QVI \eqref{eq:QVIforAlpha} satisfied by the derivative $\alpha$  
is determined as the monotone limit of the sequence $\{\alpha_n\}$ (see \eqref{eq:expansionFormulaN}) of solutions of VIs where each $\alpha_n$ satisfies
\begin{align*}
\nonumber \alpha_n \in \mathcal{K}_{\m}(\alpha_{n-1}) &: \langle A\alpha_n - d, \alpha_n - \varphi \rangle \leq 0 \quad \forall \varphi \in \mathcal{K}_\m(\alpha_{n-1}).
\end{align*}
A direct consequence of the monotonicity of $\{u^s_n\}$ allows us to conclude that $\alpha^n \nearrow \alpha$ in $V$.
\section{The maximal solution map}
This section is almost a mirror image of \S \ref{sec:min} but not entirely because our constraint set is unilateral, that is, it involves (only) the $\leq$ operator. Here, we reverse the sign of the direction term in order to enforce monotonicity of a certain sequence.

We start with $\overline{u}$ a given supersolution of $S(f,\cdot)$. Again, let us take $s \geq 0$. Observe that for any $d \in V^*$ with $d \leq 0$, $\overline{u}$ is a supersolution for $S(f+sd, \cdot)$ too: $\overline{u} \geq S(f, \overline{u}) \geq S(f+sd, \overline{u})$ by the sign on $d$. Akin to the previous section, we are going to make the standing assumption
\begin{equation}\label{eq:subsolutionForPerturbedProblem}
\text{there exists $s_0 > 0$ such that for all $s \leq s_0$, $\underline{u}$ is a subsolution for $S(f+sd, \cdot)$}
\end{equation}
so that $\mathbf{Q}(f+sd) \cap [\underline{u}, \overline{u}]$ is non-empty and $\M(f+sd)$ is well defined on $[\underline{u}, \overline{u}]$.
\begin{remark}
This assumption is more inconvenient  to fulfill than the corresponding condition \eqref{ass:supersolutionForPerturbedProblem} for the minimal solution map. Nevertheless,  if $\Phi(0) \geq 0$ and there exists an $r_0 \geq 0$ such that $f+r_0d \geq 0$, then $0$ is a subsolution satisfying \eqref{eq:subsolutionForPerturbedProblem} for sufficiently small $s$. Indeed, for $s \leq r_0$, we have $S(f+sd, 0) \geq S(f+r_0d, 0) \geq 0$.
\end{remark}
\begin{lem}\label{lem:maxMapComparison}Let $d \in -V^*_+$, 
 and  \eqref{ass:feasibleForMin} 
 hold. Then $\M(f) \geq \M(f+sd)$.
\end{lem}
\begin{proof}
Define
\begin{align*}
y_n^s &:= S(f+sd, y_{n-1}^s),\\
y_0^s &:= \overline{u}. 
\end{align*}
Taking the sequence $\{u_n\}$ from \eqref{eq:sequenceUnMax}, it follows that $u_1 \geq y_1^s$ and therefore $u_n \geq y_n^s$. Passing to the weak limit in this inequality by using Lemma \ref{lem:uNConvergenceMax}, we see that $\M(f) \geq \M(f+sd)$ (note that \eqref{eq:subsolutionForPerturbedProblem} is needed to obtain $y_n^s \geq \underline{u}$ and hence also $\lim_n y_n^s \geq \underline{u}$).
\end{proof}
It is not difficult to see that $\M(f)$ is a supersolution for $S(f+sd,\cdot)$ for non-positive $d$. This allows us to construct a perturbed sequence starting at $\M(f)$ and we obtain the next result.
\begin{prop}\label{prop:usnConvMax}Let  $d \in -V_+^*$,  
\eqref{ass:feasibleForMin} and
\begin{equation}
\text{there exists $\epsilon > 0$ such that  } \Phi\colon B_\epsilon(\M(f)) \to V \text{ is Lipschitz with Lipschitz constant $C_\Phi < C_a\slash(C_a+C_b)$}\label{ass:PhiLipBoundM}
\end{equation}
hold. With
\begin{align*}
u^s_n &:= S(f+sd, u^s_{n-1}),\\
u^s_0 &:= \M(f),
\end{align*}
we have $u^s_n \searrow \M(f+sd)$ in $V$.
\end{prop}
\begin{proof}
Firstly, as indicated in the proof of Proposition \ref{prop:usnConvMin}, \eqref{ass:PhiLipBoundM} allows us to obtain that $u_n^s \to u^s$ strongly in $V$ to some $u^s \in \mathbf{Q}(f+sd)$. We see that, using Lemma \ref{lem:maxMapComparison}, $u_1^s \geq S(f+sd, \M(f+sd)) = \M(f+sd)$, implying $u_n^s \geq \M(f+sd)$. Since $\M(f)$ is a supersolution for $S(f+sd, \cdot)$, we obtain, by Lemma \ref{lem:uNConvergenceMax}, (the following convergence  being weak) $u_n^s \searrow u^s = \M_{[\underline{u}, \M(f)]}(f+sd)$. 
Lemma \ref{lem:maxMapComparison} tells us that in fact $\M(f+sd)$ belongs to $[\underline{u}, \M(f)]$ so we must have $u^s = \M(f+sd)$ because $\M(f+sd)$ is also the largest element on $[\underline{u}, \M(f)]$.
\end{proof}
Finally, with the same reasoning as in the proof of Theorem \ref{thm:min}, due to Theorem \ref{thm:dirDiff1} and Proposition \ref{prop:usnConvMax} we can obtain the next result.
\begin{theorem}[Directional differentiability of $\M$]
For $f \in V^*_+$, let $\overline{u}$ be a supersolution for $S(f,\cdot)$ and given $d \in V^*_+$, let $\underline{u}$ satisfy \eqref{eq:subsolutionForPerturbedProblem} and $\underline{u} \leq \overline{u}$.
In addition to  
\eqref{ass:feasibleForMin} and \eqref{ass:PhiLipBoundM}, assume that
\begin{equation}
\text{$\Phi\colon V \to V$ is Hadamard directionally differentiable at $\M(f)$}\label{ass:PhiHadamardAtM}.
\end{equation}
Then the map $\M\colon V^* \to V$ is directionally differentiable at $f$ in the direction $d$:
\begin{equation}
\lim_{s \to 0^+}\frac{\M(f+sd)-\M(f)}{s}= \M'(f)(d).
\end{equation}
Furthermore, $\M'(f)(d)$ satisfies the QVI  
\begin{equation*}
\begin{aligned}
\alpha \in \mathcal{K}_\M(\alpha) &: \langle A\alpha - d, \alpha - v \rangle \leq 0 \quad \forall v \in \mathcal{K}_\M(\alpha),\\
\mathcal{K}_\M(\alpha) &:= \Phi'(\M(f))(\alpha)  + \mathcal{T}_{\mathbf K(\M(f))}(\M(f)) \cap [f-A\M(f)]^\perp.  
\end{aligned}
\end{equation*}	
\end{theorem}
In a similar way to \S \ref{sec:min}, we find that $\alpha_n \searrow \alpha$ in $V$ where
\begin{align*}
\nonumber \alpha_n \in \mathcal{K}_\M(\alpha_{n-1}) &: \langle A\alpha_n - d, \alpha_n - \varphi \rangle \leq 0 \quad \forall \varphi \in \mathcal{K}_\M(\alpha_{n-1}).
\end{align*}

\section{Examples}
In this section, we illustrate the applicability of the theory developed in this paper with some examples  and applications of QVI problems to which we can apply the above results. We start in \S \ref{sec:eg1} with a simple and somewhat academic 1D example, the purpose of which is to simply demonstrate that there are indeed non-trivial QVI problems satisfying the requirements of the above theorems. In \S \ref{sec:eg2}, we study an example where $\Phi$ is related to a solution mapping of a PDE. We conclude in \S \ref{sec:thermo} with a real world application involving a PDE coupled with a VI; this also serves to demonstrate that maps $\Phi$ like the simpler one in \S \ref{sec:eg2} have uses in physical models and applications.

Before we proceed, let us give some concrete examples of function spaces $V$ and $H$ and operators $A$ such that all of the assumptions required (stated in the introduction) are satisfied. We will also point out the examples that lead to regular Dirichlet spaces so that the improved characterisation of the critical cone (see Remark \ref{rem:dirichletSpace}) is available. 
\begin{egs}[Examples of the functional setting]\label{eg:examples}
Let $\Omega \subset \mathbb{R}^n$ be a bounded Lipschitz domain and $n \geq 1$. In the following examples, we will consider $H$ to be an $L^2$-type space and the ordering $\leq$ will be taken to be the usual one: $u \leq v$ if and only if $u \leq v$ a.e. on the set under consideration. The cone $H_+$ will be the set of a.e. non-negative functions in $H$. 
\begin{enumerate}[label=(\roman*)]
\item  Set $V=H^1_0(\Omega)$ and $H=L^2(\Omega)$ with the standard norms. Take $A$ to be a second-order linear elliptic operator of the form
\begin{equation}\label{eq:ellipticOperator}
\langle Au, v \rangle =  \sum_{i,j=1}^n \int_\Omega a_{ij}\frac{\partial u}{\partial x_i}\frac{\partial v}{\partial x_j} + \sum_{i=1}^n \int_\Omega b_{i}\frac{\partial u}{\partial x_i}v + \int_\Omega c_0 uv\qquad \forall u, v \in V,
\end{equation}
 with $a_{ij}=a_{ji} \in C^{0,1}(\bar\Omega)$, $b_i \in W^{1,\infty}(\Omega)$, $c_0 \in L^\infty_+(\Omega)$ and 
\begin{equation*}
\sum_{i,j=1}^n a_{ij}\xi_i\xi_j \geq C|\xi|^2 \quad\text{a.e.}
\end{equation*}
for some $C>0$ and such that $A$ is coercive.  The prototypical case is the selection $a_{ij} \equiv \delta_{ij}$ and $b_i \equiv c_0 \equiv 0$ which leads to $A=-\Delta$ being the Dirichlet Laplacian.
\begin{itemize}
\item $V$ is a regular Dirichlet space \cite[\S 1.2, Examples 1.2.3]{FukushimaBook} with $\xi(u,v) = \int_\Omega \grad u \cdot \grad v$.
\end{itemize}
\item Take $H$ as above but now $V=H^1(\Omega)$. Let $A$ be as in \eqref{eq:ellipticOperator} with additionally $c_0 \geq \lambda > 0$ a.e. where $\lambda$ is a constant such that $A$ is coercive. If $a_{ij} \equiv \delta_{ij}$, $b_i \equiv 0$ and $c_0 >0$ is a constant, we have $A = -\Delta + c_0\Id$ with $-\Delta$ the Neumann Laplacian.  
\begin{itemize}
\item Setting instead $H=L^2(\overline{\Omega})$ implies that the density condition in Remark \ref{rem:dirichletSpace} holds and gives rise to a regular Dirichlet form, see \cite[\S 1.6, Example 1.6.1]{FukushimaBook}. In the context of Remark \ref{rem:dirichletSpace}, we made the choice $X=\overline{\Omega}$.
\end{itemize}
\item Let $V=H^s(\Omega)$ for $s \in (0,1)$ where $H^s(\Omega)$ is the fractional Sobolev space consisting of functions in $L^2(\Omega)$ such that the quantity
\[\norm{u}{H^s(\Omega)}^2 := \int_{\Omega} u^2(x)\;\mathrm{d}x + \int_{\Omega}\int_{\Omega}\frac{|u(x)-u(y)|^2}{|x-y|^{n+2s}}\;\mathrm{d}x\;\mathrm{d}y\]
is finite. We may take $H=L^2(\overline{\Omega})$ so that $X=\overline{\Omega}$ again. This is a regular Dirichlet space (see e.g. \cite[\S 6.5, Example 6]{MR2849840}) with the form
\[\xi(u,v) = \int_{\Omega}\int_{\Omega}\frac{(u(x)-u(y))(v(x)-v(y))}{|x-y|^{n+2s}}\;\mathrm{d}x\;\mathrm{d}y.\]
An example for $A$ is $\langle Au, v \rangle = (u,v)_{H^s(\Omega)}$. 
\end{enumerate}
\end{egs}
These are examples of possible choices of function spaces and operators $A$. We now come to examples of obstacle mappings $\Phi$ and the related QVI problems.

\subsection{One-dimensional toy problem}\label{sec:eg1}
Construct an increasing, smooth function $\Phi\colon \mathbb{R} \to \mathbb{R}$ with the following properties: (a) $\Phi(0) \geq 0$, (b) for a given sequence $\{y_j\}_{j=1}^N \subset \mathbb{R}_+$ (for some $N \in \mathbb{N}$, $N >1$) such that $y_1 \leq y_2 \leq \hdots \leq y_N$, we have the existence of some $\epsilon > 0$ such that for all $j \in \{1, \hdots, N\}$, we have $\Phi|_{B_\epsilon(y_j)} \equiv y_j$.

Set $f:=\max(0, Ay_1, \hdots, Ay_N)$. Then $Ay_j -f \leq 0$ and if $v$ is such that $v \leq \Phi(y_j)$ then $y_j-v = \Phi(y_j)-v \geq 0$, hence each $y_j \in \mathbf{Q}(f)$ is a solution.

Taking $d \geq 0$, it is evident that $y_1$ is a subsolution for $S(f,\cdot)$ and $S(f+sd,\cdot)$, and $A^{-1}(f+d)$ is a supersolution (see Example \ref{eg:ofSupersolution}) with $\m(f) =y_1$ the minimal on $[y_1, A^{-1}(f+d)]$. Clearly, all assumptions of Theorem \ref{thm:min} are satisfied and we get the directional differentiability of $\m$ in all non-negative directions.
\subsection{Obstacle maps related to inverse of elliptic operators}\label{sec:eg2} Let $L\colon V \to V^*$ be a bounded, linear\footnote{This can be readily generalised to nonlinear, Fr\'echet differentiable operators $L$ with appropriate changes to the calculations that follow.}, coercive (with coercivity constant $K_a$) and T-monotone operator and take a Hadamard differentiable map $g\colon H \to V^*$ such that
\begin{enumerate}[label=(\roman*)]
\item $g(0) = 0$,
\item $g\colon B_\epsilon(0) \to V^*$ is Lipschitz for some $\epsilon>0$ with Lipschitz constant $C_g$ satisfying $C_g \leq K_aC_\Phi$ where $C_\Phi$ is as in \eqref{ass:PhiLipBoundm},
\item $u \leq v$ in $H$ implies that $g(u) \leq g(v)$ (i.e., $g$ is increasing).
 \end{enumerate} Define $\Phi\colon H \to V$ by $\phi =\Phi(u)$ if and only if
\[L\phi = g(u).\]
That is, $\Phi(u)$ is the solution of a PDE. It is standard that this is well defined: for every $u \in H$, there exists a unique $\phi \in V$ satisfying this equation by the Lax--Milgram lemma.

If $u \leq v$, denoting $\phi_u = \Phi(u)$ and $\phi_v = \Phi(v)$, we have
\[\langle L(\phi_u)-L(\phi_v), (\phi_u-\phi_v)^+ \rangle = \langle g(u)-g(v), (\phi_u-\phi_v)^+\rangle \leq 0\]
since $g$ is increasing. By T-monotonicity of $L$, $\phi_v \leq \phi_u$, showing that $\Phi$ is an increasing map. It is also not difficult to see that $\Phi(0)=0$.

Concerning differentiability, take $u, h \in V$. Denoting $\varphi := \Phi(u+sh) = L^{-1}g(u+sh)$ and $\phi:=\Phi(u)=L^{-1}g(u)$, we see that, using the Hadamard differentiability of $g$,
\[\varphi = L^{-1}(g(u) + sg'(u)(h) + o(s)) = \phi + sL^{-1}g'(u)(h) + L^{-1}o(s),\]
hence $\Phi$ is Hadamard differentiable with $\Phi'(u)(h) = L^{-1}g'(u)(h)$ given as the unique solution of the equation 
\[L\delta = g'(u)(h).\]
Now, take $f \in V^*_+$. If $v \in V$ is such that $v \leq \Phi(0)=0$ then $\langle A0-f, 0-v \rangle = \langle f,v \rangle  \leq 0$, i.e., $0 \in \mathbf{Q}(f)$.

Let $d \in V^*_+$. From Example \ref{eg:ofSupersolution},  $A^{-1}(f+d)$ is a supersolution with $A^{-1}(f+d) \geq 0$. It follows that $\mathbf{Q}(f)\cap[0, A^{-1}(f+d)]$ is non-empty and there exist minimal and maximal solutions on this interval. Let us point out that if in addition
\[A^{-1}f \leq L^{-1}g(A^{-1}f),\]
it is easy to see that $A^{-1}f \in \mathbf{Q}(f)$ is a solution (and not just a supersolution of $S(f,\cdot)$). This condition is satisfied if for example, $L \equiv A$ and $g(A^{-1}f) \geq f$, as can be seen via the weak comparison principle.

Now, take $u, v \in B_\epsilon(0)$ for some $\epsilon$. Then $\Phi(u)-\Phi(v) = L^{-1}g(u) - L^{-1}g(v)$ and hence
\begin{align*}
\norm{\Phi(u)-\Phi(v)}{V} &\leq K_a^{-1}\norm{g(u)-g(v)}{V^*}\\
&\leq K_a^{-1}C_g\norm{u-v}{V}
\end{align*}
which, by assumption (ii) on $g$ above, shows that \eqref{ass:PhiLipBoundm} is satisfied. Therefore, we again have the directional differentiability of $\m$ for all $d \in V^*_+$.

\subsection{Application to thermoforming}\label{sec:thermo}
We study now an application in thermoforming that we set out in \cite{AHR}; the background and introductory details below are slightly revised from there but afterwards we will also improve and expand previous results and show that the minimal solution map associated to this example is differentiable. 

Thermoforming is a process used to manufacture products by heating a membrane or plastic sheet to its pliable temperature and then forcing the membrane (by means of vacuum or gas pressure) onto a mould which is typically made of aluminium or an alloy; this makes the membrane deform and take on the shape of the mould.  The process is applied to form structures on a variety of physical scales from car panels to microfluidic structures (e.g. channels on the range of micrometers). The importance of the applications and the necessity of precision of the details in this process has sparked research into its modelling and accurate numerical simulation, e.g., \cite{Munro2001, Warby2003}. 

Mathematically, the contact problem associated with the membrane and the mould can be described as a variational inequality problem (assuming perfect sliding of the membrane with the mould as described in \cite{Whiteman2000}). However, when the heated sheet is forced into contact with the mould, a complex phenomenon takes place: in principle, the mould and the plastic sheet are at different temperatures (the mould may be cold relative to the membrane), which triggers a heat transfer process with difficult-to-predict consequences (see for example \cite{Lee2001}, e.g., it changes the polymer viscosity). In practice, the thickness of the thermoformed piece can be controlled locally by the mould structure and its initial temperature distribution \cite{Lee2001} and the non-uniform temperature distribution of the polymer sheet has substantial influence on the results \cite{Nam2000}. Indeed, a commonly-used material for the mould is aluminum and large heat fluctuations create a substantial difference in the size; aluminum has a relatively high thermal expansion volumetric coefficient and this implies that there is a dynamic change in the obstacle (the mould) as the polymer sheet is forced in contact with it. This suggests that the model can be formulated as a compliant obstacle-type problem (such as the one described in \cite{Outrata1998}) and hence the overall process is a QVI with underlying nonlinear PDEs determining the heat transfer and the volume change in the obstacle.

We consider this compliant obstacle behavior whilst making simplifying assumptions in order to study a tractable yet meaningful model.

\subsubsection{The mathematical model}
We restrict the analysis to the 1D case for the sake of simplicity. Let $\Phi_0\colon [0,1] \to \mathbb{R}^+$ be the (parametrised) mould shape that we wish to reproduce through a sheet or membrane. The membrane lies below the mould and is pushed upwards through some mechanism (as mentioned, usually through vacuum and/or air pressure) denoted as $f$.  We make the following fundamental physical assumptions to simplify the oncoming study:
\begin{enumerate}
\item[(i)] the temperature for the membrane $u$ is constant and prescribed,
\item[(ii)] the growth of the mould $\Phi(u)$ is affine with respect to its temperature $T$, 
\item[(iii)] the temperature (denoted by $T$, as above) of the mould (a) is subject to diffusion, convection and boundary conditions arising from the insulated boundary and, (b) depends on the vertical distance between the mould $\Phi(u)$ and the membrane $u$.
\end{enumerate}
Despite the fact that the thermoforming process is an evolutionary process, the setting described by the above assumptions is appropriate for one time step in the time discretisation and fits the mathematical framework of the present paper.  

Define $V=H^1(0,1)$, $H=L^2(0,1)$ and let $A=-\Delta + I$ 
where $-\Delta$ is the Neumann Laplacian. The system we consider is the following:
\begin{subequations}\label{eq:thermoformingSystem}
\begin{align}
u \in V : u \leq \Phi(u), \quad \langle Au-f, u-v \rangle &\leq 0\quad \forall v \in V : v \leq \Phi(u),\label{eq:QVIforu}\\
kT-\Delta T &= g(\Phi(u)-u) &&\text{on $[0,1]$},\label{eq:pdeForT}\\
\partial_\nu T &= 0 &&\text{on $\{0, 1\}$},\label{eq:bcForT}\\
\Phi(u) &= \Phi_0 + LT &&\text{on $[0,1]$},\label{eq:mould}
\end{align}
\end{subequations}
where 
$f \in H_+$, $k>0$ is a constant, $\Phi_0 \in V$, $L\colon V \to V$ 
is a bounded linear operator such that, with $\Omega:=[0,1]$,
\[
\text{for every $\Omega_0 \subset \Omega$, if $u \leq v$ a.e. on $\Omega_0$ then $Lu \leq Lv$ a.e. on $\Omega_0$,}\]
and  $g\colon \mathbb{R} \to \mathbb{R}$ is decreasing and $C^2$ with $g(0)=M > 0$ a constant, $0 \leq g \leq M$ with $g'$ bounded\footnote{Under these circumstances, $g$ maps $V$ into $V$ \cite[Theorem 1.18]{MR1207810}} with 
\begin{equation}\label{eq:assOngForThermo}
g' \equiv 0 \text{ on } (1,\infty),
\end{equation}
and $g''$ bounded from above. Thus when the membrane and mould are in contact or are close to each other, there is a maximum level of heat transfer onto the mould, whilst when they are sufficiently separated, there is very little or no heat exchange. An example of $g$ to have in mind is a smoothing of the function
\begin{equation}\label{eq:exampleg}
\tilde g(r) = \begin{cases}
1 &: \text{if $r \leq 0$}\\
1-r &: \text{if $0 < r < 1$}\\
0 &: \text{if $r \geq 1$}.
\end{cases}
\end{equation}
\begin{remark}
The system \eqref{eq:thermoformingSystem} is derived as follows. Consideration of the potential energy of the membrane will show that $u$ solves the inequality \eqref{eq:QVIforu} \cite{Whiteman2000} with the novel QVI nature resulting from assuming that heat transfer occurs between the membrane and mould (the membrane and the mould modify each other). If we let $\hat T\colon \Gamma \to \mathbb{R}$ be the temperature of the mould defined on the curve
\[\Gamma :=\{(r,\Phi(u)(r)) : r \in [0,1]\} \subset \mathbb{R}^2,\]
our modelling assumptions directly imply that $\hat T$ solves the PDE
\begin{equation*} 
\begin{aligned}
k\hat T(x)-\Delta_\Gamma \hat T(x) &= g(x_2-u(x_1)) &&\text{for $x=(r,\Phi(u)(r)) \in \Gamma$},\\
\frac{\partial \hat T}{\partial \nu} &= 0 &&\text{on $\partial\Gamma$, 
}
\end{aligned}
\end{equation*}
where $x=(x_1,x_2)$ and $-\Delta_\Gamma$ is the Laplace--Beltrami operator. Reparametrising by $T(r) = \hat T(r, \Phi(r))$ and simplifying the Laplace--Beltrami operator leads to \eqref{eq:pdeForT}, see the appendix of \cite{AHR} for the details.
\end{remark}

\subsubsection{Properties for the system}

It can be useful to write the equation for $T$ explicitly in terms of $u$ by plugging  \eqref{eq:mould} into \eqref{eq:pdeForT}:
\begin{equation}\label{eq:equationForT}
\begin{aligned}
kT-\Delta T &= g(LT+\Phi_0 -u) &&\text{on $[0,1]$},\\
\partial_\nu T &= 0 &&\text{on $\{0, 1\}$ }.
\end{aligned}
\end{equation}
According to \cite[Lemma 16]{AHR}, for every $u \in H$, there exists a unique solution $T \in V$ to this equation. This formulation allows us to deduce the following properties for $\Phi$.
\begin{lem}[{\cite[Lemmas 17 and 18]{AHR}}]
The map $\Phi\colon H \to V \subset H$ is increasing and $\Phi(0) \geq 0$ a.e.
\end{lem}
As before, we choose $\underline{u}=0$ as subsolution and $\overline{u}=A^{-1}(f+d)$ as supersolution. Using these results along with the Birkhoff--Tartar theorem, we have the following existence result.
\begin{theorem}[{\cite[Theorem 7]{AHR}}]\label{thm:existenceForThermoformingSystem}
There exists a solution $(u, T, \Phi(u))$ to the system \eqref{eq:thermoformingSystem}.
\end{theorem}
Testing \eqref{eq:equationForT} with the solution $T$ and using the fact that $g \in L^\infty(\mathbb R)$, we have the \emph{a priori} bound (independent of $\Phi_0$, $L$ and $u$)
\[\norm{T}{V} \leq \norm{g}{\infty}\min(1,k)^{-1} =: C^*.\]
This next result improves \cite[Lemma 19]{AHR}, which needed a smallness condition related to various constants appearing in the model for the conclusion to hold.
\begin{lem}
The map $\Phi\colon V \to V$ is completely continuous.
\end{lem}
\begin{proof}
Let $u_n \to u$ in $V$ and denote by $T_n$ the associated temperature, i.e.,
\begin{align*}
kT_n - \Delta T_n &= g(LT_n +\Phi_0 - u_n),\\
\partial_\nu T_n &= 0.
\end{align*}
Using the estimate derived above, we get the uniform bound
\[\norm{T_n}{V} \leq C^*.\]
Thus, for a subsequence that we shall relabel, $T_n \weaklyto T$ in $V$ to some $T$. The compact embedding $V \ctsCompact H$ implies that $T_n \to T$ in $H$, and by the continuity of $g$ and $L$ and the Dominated Convergence Theorem, $g(LT_n+\Phi_0-u_n) \to g(LT+\Phi_0-u)$ in $H$, thus we find that $T$ satisfies the expected limiting equation \eqref{eq:equationForT}. Since solutions of the latter equation are unique, the entire sequence $\{T_n\}$ converges, not just a subsequence. The claim follows (by definition of $\Phi$, see \eqref{eq:mould}).
\end{proof}
Complete continuity of $\Phi$ is much stronger than what is necessary to enter into the framework of the theory presented in the previous sections but it is a useful property that typically holds in these types of models involving PDEs.

Regarding the differentiability of $\Phi$ and a characterisation of the derivative, we have the following, which can be proved via the implicit function theorem.
\begin{theorem}[{\cite[Theorem 8]{AHR}}]\label{thm:diffOfPhiThermo}The map $\Phi\colon V \to V$ is Fr\'echet differentiable at any solution $u$. Furthermore, given $d \in V$, denoting by $\delta \in V$ the solution of
\begin{align*}
(k-\Delta)\delta - g'(\Phi(u)-u)L\delta &= g'(\Phi(u)-u)d,\\
\partial_\nu \delta &= 0,
\end{align*}
we have $\Phi'(u)(d) = -L\delta $.
\end{theorem}
It remains for us to verify the Lipschitz condition \eqref{ass:PhiLipBoundm}. 
\begin{lem}
Let $\Phi_0 > 1+ C_a^{-1}K\norm{f}{V^*}$ where ${\inc}$ denotes the injectivity constant in the embedding $V \cts L^\infty(\Omega)$. Then the condition \eqref{ass:PhiLipBoundm} is satisfied.
\end{lem}
\begin{proof}

Test the inequality satisfied by $\m(f)$ with $0$, which is feasible since $0 \leq \Phi(0) \leq \Phi(\m(f))$ with the latter inequality by the increasing property of $\Phi$. This immediately leads to the bound
\[\norm{\m(f)}{V} \leq C_a^{-1}\norm{f}{V^*}.\]
For some $\rho > 0$, take $v \in B_{\rho\slash \inc}(\m(f))$, hence $\norm{v-\m(f)}{L^\infty(\Omega)} \leq \rho$ . We have, using the fact that $T \geq 0$ and the local increasing property of $L$,
\begin{align*}
\Phi(v)-v &= \Phi_0 + LT(v)- v\\
 &\geq \Phi_0  - v\\
&\geq \Phi_0 - \m(f) - \rho\\
&\geq \Phi_0 - C_a^{-1}K\norm{f}{V^*}- \rho.
\end{align*}
It follows that if $\Phi_0 > 1 + C_a^{-1}K\norm{f}{V^*}$, we can choose $\rho$ so that the above exceeds $1$ and by the condition \eqref{eq:assOngForThermo} on $g'$, we have
\[g'(\Phi(v) -v) = 0.\]
Then for such $v \in B_{\rho\slash \inc}(\m(f))$, we see from Theorem \ref{thm:diffOfPhiThermo} that  $L^{-1}\Phi'(u)(d) =0$ for all directions $d$, hence \eqref{ass:PhiLipBoundm} is satisfied (via Remark \ref{rem:2}).
\end{proof}
Applying now Theorem \ref{thm:min}, we obtain the directional differentiability of $\m$ at $f \in H_+$ in the direction $d \in V^*_+$.
\section{Conclusion}
In this paper, we have given the first directional differentiability results for minimal and maximal solution maps associated to QVIs and have demonstrated the theory with some examples. In conclusion, let us note that the formulation and analysis (such as existence and stationarity theory) of optimal control problems with these types of control-to-state maps (as discussed in \S \ref{sec:aims}) requires additional effort and care as well as further continuity results. We aim to address these issues in future work.
\bibliographystyle{abbrv}
\bibliography{QVILatestBibEvenLater}
\end{document}